\documentclass[reqno, 11pt]{amsart}
\usepackage{amsmath}
 \usepackage{amssymb}
\usepackage{amsthm}
\usepackage{times}
\usepackage{latexsym}
\usepackage[mathscr]{eucal}

\numberwithin{equation}{section}
 
  \newtheorem{theorem}{Theorem}[section]
  \newtheorem{proposition}[theorem]{Proposition}
  \newtheorem{lemma}[theorem]{Lemma}
  \newtheorem{corollary}[theorem]{Corollary}

  \newtheorem{definition}[theorem]{Definition}

\title[Sasakian Finsler structures on pulled-back bundle]{Sasakian Finsler structures on pulled-back bundle}
\author[ Fortun\'{e} Massamba and Salomon Joseph Mbatakou]{ Fortun\'{e} Massamba*, \, Salomon Joseph Mbatakou** }

\newcommand{\acr}{\newline\indent}

\address{\llap{*\,}  School of Mathematics, Statistics and Computer Science\acr
 University of KwaZulu-Natal\acr
 Private Bag X01, Scottsville 3209\acr
South Africa}
\email{massfort@yahoo.fr, Massamba@ukzn.ac.za}

\address{\llap{**\,}  School of Mathematics, Statistics and Computer Science\acr
 University of KwaZulu-Natal\acr
 Private Bag X01, Scottsville 3209\acr
South Africa}
\email{mbatakou@gmail.com, mbatakous@ukzn.ac.za}

\thanks{}

\subjclass[2010]{Primary 53D10; Secondary 53C60}

\keywords{Almost contact structure; Contact structure; Finslerian manifold.}

\begin{document}
\maketitle

 \begin{abstract}
 Under a pulled-back approach given in \cite{AE} and firstly presented in \cite{AA} , we introduce, in this paper, the concepts of almost contact  and normal almost contact Finsler structures on the pulled-back bundle.  Properties of structures partly Sasakians are studied.  Using the hh-curvature tensor of Chern connection given in \cite{AA}, we obtain some characterizations of horizontally Finslerian K-contact structures via the horizontal Ricci tensor and  the  flag curvature.
 \end{abstract}

 \section {Introduction}
 Let $(M,F)$ be a Finsler manifold and $TM_{0}$ be its slit tangent bundle. There exist, in the literature, several frameworks for the study of 
Finsler geometry. For instance, an approach through the double tangent bundle $TTM_{0}$ (Grifone's approach), an approach via the vertical subbundle of $TTM_{0}$ (Bejancu-Farran, Abate-Patrizio, etc.) and the pulled-back bundle approach (Bao-Chern-Shen,$\cdots$). The latter motivates this paper. In fact, it is the most natural approach, because it facilitates the analogy with the Riemannian geometry. 
The key idea of this study is to construct, like in Riemannian case, the Sasakian structures in Finsler geometry. Let $\pi$ be a canonical submersion from $T M_{0}$ onto $M$. The pulled-back bundle $\pi^{\ast}TM$, which is nothing but a collection of fibers of $TM$ on $TM_{0}$, offers an adequate framework of this study.    

The paper is organized as follows. In Section \ref{section1}, we recall some basic definitions and necessary geometric concepts that are used 
throughout this paper. We also define in an adapted tensorial formalism, an almost contact Finsler structure $(\phi, \eta, \xi)$, on pulled-back bundle $\pi^{\ast}TM$. We introduce in Section \ref{Almost} almost contact and contact metric Finsler structures, and obtain some characterizations.
We define horizontally and vertically Sasakian Finsler structures and K-contact Finsler structures. Under some condition, we prove that contact metric 
Finsler structure on $\pi^{\ast}TM $,  is horizontally K-contact and vertically K-contact. Using the Chern connection, we determine the covariant derivative of $\phi$ and $\xi$, in any direction, thereby generalizing the Riemannian case for the Levi-civita connection. In the same section, we discuss of some aspects of hh-curvature with respect to the Chern connection, of contact metric Finsler pulled-back bundle. We obtain the horizontal Ricci $(1,1;0)$-tensor relatively to $\xi$ and his horizontal representative. Finally, with the aid of the flag curvature with transverse edge $\xi$, we obtain a Finslerian analogous characterization's result of Hatakeyama, Ogawa and Tanno's \cite{HOT}, on K-contact metric  structures.

 \section{Preliminaries}\label{section1}

Let $\pi: TM \rightarrow M $ be a tangent bundle of connected smooth Finsler manifold $M$ of odd-dimension $m=2n+1$.
We denote by $v= (x,y)$ the points in $TM$ if $y\in \pi^{-1}(x)= T_{x}M$. We denote by $O(M)$ the zero section of $TM$,
and by $TM_{0}$ the slit tangent bundle $ TM\setminus O(M)$. We introduce a coordinate system on $TM$ as follows.
Let $U\subset M$ be an open set with local coordinate $(x^{1},...,x^{m})$.
By setting $v= y^{i}\frac{\partial}{\partial x^{i}}$ for every $v\in\pi^{-1}(U) $, we introduce a
local coordinate $(x,y)=(x^{1},...,x^{m},y^{1},...,y^{m})$ on $\pi^{-1}(U)$.

\begin{definition}{\rm
 A function $F: TM \rightarrow [0,+\infty[$ is called a Finsler structure or Finsler metric on $M$ if:
  \begin{enumerate}
  \item[(i)] $F\in C^{\infty}(TM_{0})$,
\item[(ii)]  $F(x,\lambda y)= \lambda F(x,y)$, for all $ \lambda > 0$,
\item[(iii)]  The $m\times m$ Hessian matrix $( g_{ij})$, where
$
\displaystyle\label{ft} g_{ij}:= \frac{1}{2}(F^{2})_{y^{i}y^{j}}
$
  is positive-definite at all $(x,y)$ of $TM_{0}$.
  \end{enumerate}
  }
\end{definition}
The pair $(M, F )$ is called \textit{Finsler manifold}. The pulled-back bundle  $\pi^{\ast}TM$ is a vector bundle over the slit tangent bundle $ TM_{0}$,
defined by
\begin{equation}
\pi^{\ast}TM:= \{(x,y,v)\in TM_{0}\times TM: v \in T_{\pi(x,y)}M \}.
\end{equation}
 By the objects (\ref{ft}), the pulled-back vector bundle  $\pi^{\ast}TM$ admits a natural Riemannian metric
 \begin{equation}
 g:= g_{ij}dx^{i}\otimes dx^{j}.
 \end{equation}
 This is the Finslerian fundamental tensor in a sense that we will specify later. Likewise, there is the Finslerian Cartan tensor
 \begin{equation}\label{Atensor}
 A= A_{ijk}dx^{i}\otimes dx^{j}\otimes dx^{k},\;\;\mbox{with}\;\; A_{ijk}:= \frac{F}{2}\frac{\partial g_{ij}}{\partial y^{k}}.
 \end{equation} 
Note that, with a slight abuse of notation, $\frac{\partial}{\partial x^{i}}$ and $dx^{i}$ are regarded as sections
 of $\pi^{\ast}TM$ and $\pi^{\ast}T^{\ast}M$, respectively.

 Now we will give some geometric tools for understanding the intrinsic formulation of geometric objects that we used in this paper.

 It is well know that, the kernel of $\pi_{\ast}$, spanned the vertical subbundle $\mathcal{V}$ of $TTM_{0}$. An Ehresmann 
 connection is the choice of the horizontal complementary  $\mathcal{H}\subset TTM_{0}$ such that 
\begin{equation}\label{dec}
TTM_{0} =   \mathcal{H} \oplus   \mathcal{V}.
\end{equation}
%

In this paper, we shall consider the choice of Ehresmann connection which arises from the Finsler structure $F$,
constructed as follows. Recall that \cite{AF} every Finslerian structure $F$ induces a spray
$$
G= y^{i}\frac{\partial}{\partial x^{i}}- 2 G^{i}(x,y)\frac{\partial}{\partial y^{i}},
$$
in which the spray coefficients  $G^{i}$ are defined by
\begin{equation}
 G^{i}(x,y):= \frac{1}{4}g^{il}\left[2 \frac{\partial g_{jl}}{\partial x^{k}}(x,y)-
 \frac{\partial g_{jk}}{\partial x^{l}}(x,y)\right]y^{j}y^{k},
\end{equation}
where the matrix $(g^{ij})$ means the inverse of $(g_{ij})$.

Define a $\pi^{\ast}TM$-valued smooth form on $TM_{0}$ by
\begin{equation}\label{teta}
\theta = \frac{\partial}{\partial x^{i}}\otimes \frac{1}{F}(dy^{i}+N_{j}^{i}dx^{j}),
\end{equation}
where functions  $N^{i}_{j}(x,y)$ are given by
$ 
\displaystyle N^{i}_{j}(x,y):= \frac{\partial G^{i}}{\partial y^{j}}(x,y).
$ 
This $\pi^{\ast}TM$-valued smooth form  $\theta$ is globally well defined on $TM_{0}$ \cite{AC}.

 By the form $\theta$, defined in (\ref{teta}) which is called Finsler-Ehresmann form, we can define a Finsler-Ehresmann connection  as follow.
\begin{definition}{\rm
 A Finsler-Ehresmann connection of the submersion $\pi:TM_{0}\rightarrow M$ is the subbundle $\mathcal{H}$ of $TTM_{0}$
 given by $ \mathcal{H} = \ker \theta $, where $\theta : TTM_{0}\rightarrow \pi^{\ast}TM$ is the bundle morphism defined in  (\ref{teta}),
 and which is complementary to the vertical subbundle $\mathcal{V}$.}
\end{definition}
It is well know that, $\pi^{\ast}TM$ can be naturally identified with the horizontal subbundle $\mathcal{H}$ and the vertical one
$\mathcal{V}$ \cite{AE}.
Thus, any section $\overline{X}$ of $\pi^{\ast}TM$ is considered as a section of $\mathcal{H}$ or a section of $\mathcal{V}$.
We denote by $\overline{X}^{H}$ and $\overline{X}^{V}$ respectively, the section
of $\mathcal{H}$ and the section of $\mathcal{V}$ corresponding to $\overline{X}\in \Gamma (\pi^{\ast}TM)$:
\begin{equation}\label{tr}
 \overline{X}= \frac{\partial}{\partial x^{i}}\otimes \overline{X}^{i}\in \pi^{\ast}TM
\Longleftrightarrow \overline{X}^{H}=\frac{\delta}{\delta x^{i}}\otimes \overline{X}^{i} \in \Gamma(\mathcal{H}),
\end{equation}
and
\begin{equation}
\overline{X}= \frac{\partial}{\partial x^{i}}\otimes \overline{X}^{i}\in \pi^{\ast}TM
\Longleftrightarrow \overline{X}^{V}=F\frac{\partial}{\partial y^{i}}\otimes \overline{X}^{i} \in \Gamma(\mathcal{V}),
\end{equation}
where
$$
\{F\frac{\partial}{\partial y^{i}}:= (\frac{\partial}{\partial x^{i}})^{V}\}_{i= 1,...,m}\;\;\mbox{and}\;\;\{\frac{\delta}{\delta x^{i}}:=\frac{\partial}{\partial x^{i}} - N_{j}^{i}\frac{\partial}{\partial y^{i}} = (\frac{\partial}{\partial x^{i}})^{H}\}_{i= 1,...,m},
$$
are the vertical and horizontal lifts of natural local frame field $\{ \frac{\partial}{\partial x^{1}},..., \frac{\partial}{\partial x^{m}}\}$ with respect to the Finsler-Ehresmann connection $\mathcal{H}$, respectively. 
\begin{proposition} \label{PropIso}\cite{AE}
 The bundle morphism $\pi_{\ast}$ and $\theta$ satisfy
 $$
 \pi_{\ast}(\overline{X}^{H})= \overline{X}, \;\;\; \pi_{\ast}(\overline{X}^{V})= 0 \; \; \mbox{and}\;\;  \theta(\overline{X}^{H})= 0, \;\;\;  \theta(\overline{X}^{V})= \overline{X}
 $$ 
for every $\overline{X} \in \Gamma(\pi^{\ast}TM)$.
\end{proposition}
The Proposition \ref{PropIso} means that $\mathcal{H} T M_{0}$, as well as $\mathcal{V} T M_{0}$, can be naturally identified with the bundle $\pi^{*} T M$, that is,
\begin{equation}\label{Isomorphs}
 \mathcal{H} T M_{0}\cong \pi^{*} T M\;\;\mbox{and}\;\; \mathcal{V} T M_{0}\cong \pi^{*} T M.
\end{equation} 

Next, we recall the definition of the Chern connection on the pulled-back bundle which is going to be used throughout the paper. This connection is symmetric but not always compatible with the metric of the underlying manifold.
\begin{theorem} \cite{mb1}\label{th1}
 Let $(M,F)$ be a Finsler manifold, $g$ a fundamental tensor of $F$ and $\theta$ the vector form defined in (\ref{teta}).  There exist a unique linear connection $\nabla$ on $\pi^{\ast}TM$ such that, for all $X,Y \in \Gamma(TTM_{0})$ and $\overline{Y}, \overline{Z} \in \Gamma(\pi^{\ast}TM)$, we have,
 \begin{enumerate}
  \item[(a)] Symmetry
  \begin{equation}\label{et1}
  \nabla_{X}\pi_{\ast}Y - \nabla_{Y}\pi_{\ast}X = \pi_{\ast}[X,Y],
  \end{equation}\label{A-comp}
   \item[(b)] Almost $g$-compatibility
   \begin{equation}\label{et2}
 (\nabla_{X} g)(,\overline{Y},\overline{Z})= 2A(\theta(X),\overline{Y},\overline{Z}),
   \end{equation}
 where $A$ is the Cartan tensor defined in (\ref{Atensor}).
 \end{enumerate}
\end{theorem}  
A tensor field $T$ of type $(p_{1},p_{2};q)$ on $(M,F)$ is a map:
  $$T:\Gamma((\pi^{\ast}TM )^{p_{1}})\times\Gamma((TTM_{0})^{p_{2}})\longrightarrow \Gamma((\pi^{\ast}TM)^{q}),$$
 which is $C^{\infty}(TM_{0})$-linear in each arguments. 

In the sequel, we denote by $\mathcal{T}^{(p,q;r)}$ the space of Finslerian tensor of type $(p,q;r)$. By the Finsler-Ehresmann  connection, the exterior differential
\begin{equation}
d: \mathcal{T}^{(p,0;0)} \longrightarrow \mathcal{T}^{(p+1,0;0)}
\end{equation} is decomposed into the horizontal part $d^{H}$ and vertical part $d^{V}$
according to the decomposition (\ref{dec}). Here $d^{H}$ is exterior differential along $\mathcal{H}$ and $d^{V}$ the one along $\mathcal{V}$. These exterior differentials are related to the Chern connection as follows.
\begin{definition}[Horizontal and vertical exterior differential]{\rm
Let $T$ be a $(p,0;0)$-Finslerian tensor. Then the horizontal and vertical exterior differentials $d^{H}T$ and $d^{V}T$ of $T$
are the $(p+1,0;0)$-Finslerians tensor given, respectively, by
\begin{align} \label{dex}
& d^{H}T(\overline{X}_{1},...,\overline{X}_{p+1})  =  \frac{1}{p+1}
 \sum_{i=1}^{p+1} (-1)^{i+1}\nabla_{\overline{X}_{i}^{H}}T(\overline{X}_{1},..., \breve{\overline{X}}_{i},...,\overline{X}_{p+1})\nonumber\\
&+ \frac{1}{p+1}\sum_{1\leq i<j \leq p+1} (-1)^{i+j} T(\pi_{\ast}[\overline{X}_{i}^{H}, \overline{X}_{j}^{H}],
\overline{X}_{1},..., \breve{\overline{X}}_{i},...,\breve{\overline{X}_{j}},..., \overline{X}_{p+1} ),
 \end{align}
 and
 \begin{align}
  & d^{V}T(\overline{X}_{1},...,\overline{X}_{p+1}) = \frac{1}{p+1}
 \sum_{i=1}^{p+1} (-1)^{i+1}\nabla_{\overline{X}_{i}^{V}}T(\overline{X}_{1},..., \breve{\overline{X}}_{i},...,\overline{X}_{p+1}  )\nonumber\\
&+ \frac{1}{p+1}\sum_{1\leq i<j \leq p+1} (-1)^{i+j} T(\theta([\overline{X}_{i}^{V}, \overline{X}_{j}^{V}]),
\overline{X}_{1},..., \breve{\overline{X}}_{i},...,\breve{\overline{X}_{j}},..., \overline{X}_{p+1}).\nonumber
\end{align}
}
\end{definition}
In a similar way,  for all  $\overline{X} \in \Gamma(\pi^{\ast}TM)$ the Lie derivative $\mathcal{L}_{\overline{X}}$ of the $(p,0;r)$-Finslerian
tensor, is decomposed into the horizontal $\mathcal{L}_{\overline{X}}^{H}$ and vertical $\mathcal{L}_{\overline{X}}^{V}$. Note that, for $f \in C^{\infty}(TM_{0})$, and  $\overline{Y} \in \Gamma(\pi^{\ast}TM)$, we have
\begin{align}
 \mathcal{L}_{\overline{X}}^{H}f = \overline{X}^{H}(f), \;\; \mathcal{L}_{\overline{X}}^{V}f = \overline{X}^{V}(f),\;\;
 \mathcal{L}_{\overline{X}}^{H}\overline{Y} = \pi_{\ast}[\overline{X}^{H},\overline{Y}^{H}],\;\;\; \mathcal{L}_{\overline{X}}^{V}\overline{Y} = \theta[\overline{X}^{V},\overline{Y}^{V}].\nonumber
\end{align}
Now more generaly  for the $(p,0;r)$-Finslerian
tensor we have,
\begin{definition} [Horizontal and vertical Lie derivative]{\rm
Let $\overline{X} \in \Gamma(\pi^{\ast}TM)$, and $T \in \mathcal{T}^{(p,0;r)}$,
the horizontal and vertical Lie derivative of $T$ relative to  $\overline{X}$ are given
respectively by
\begin{align}\label{rlh}
& \left(\mathcal{L}_{\overline{X}}^{H}T\right)(\overline{Y_{1}},..., \overline{Y_{p}}, \alpha_{1},..., \alpha_{r})  =
 \overline{X}^{H}\left(T(\overline{Y_{1}},..., \overline{Y_{p}},\alpha_{1},..., \alpha_{r})\right) \nonumber\\
&-  \sum_{i=1}^{p}T(\overline{Y_{1}},...,\mathcal{L}_{\overline{X}}^{H}\overline{Y_{i}},...,\overline{Y_{p}},\alpha_{1},..., \alpha_{r})\nonumber\\
&-  \sum_{j=1}^{p}T(\overline{Y_{1}},...,\overline{Y_{i}},...,\overline{Y_{p}},\alpha_{1},...,
\mathcal{L}_{\overline{X}}^{H}\alpha_{i},... \alpha_{r}),
\end{align}
and
\begin{align}\label{rlv}
&\left(\mathcal{L}_{\overline{X}}^{V}T\right)(\overline{Y}_{1},..., \overline{Y}_{p},\alpha_{1},..., \alpha_{r})=
 \overline{X}^{V}\left(T(\overline{Y}_{1},...,\overline{Y}_{p},\alpha_{1},..., \alpha_{r})\right) \nonumber\\
&-  \sum_{i=1}^{p}T(\overline{Y}_{1},...,\theta[\overline{X}^{V},\overline{Y}_{i}^{V}],...,\overline{Y}_{p},\alpha_{1},..., \alpha_{r})\nonumber\\
&- \sum_{j=1}^{p}T(\overline{Y_{1}},...,\overline{Y_{i}},...,\overline{Y_{p}},\alpha_{1},...,
\mathcal{L}_{\overline{X}}^{V}\alpha_{i},... \alpha_{r}),
\end{align}
where
\begin{align}
 (\mathcal{L}_{\overline{X}}^{H}\alpha_{i})(\overline{Y}) & =  \overline{X}^{H}\alpha_{i}(\overline{Y})- \alpha_{i}(\pi_{\ast}[\overline{X}^{H}, \overline{Y}^{H}]),\nonumber\\
 \mbox{and}\;\; (\mathcal{L}_{\overline{X}}^{V}\alpha_{i})(\overline{Y}) &=  \overline{X}^{V}\alpha_{i}(\overline{Y})- \alpha_{i}(\theta[\overline{X}^{V}, \overline{Y}^{V}]).\nonumber
 \end{align}
}
\end{definition}
The covariant derivative of the Chern connection is defined by the following.
\begin{definition}[Covariant Chern derivative]{\rm
 Let  $T$ a Finslerian tensor of type  $(p,q;r)$  and let $X \in \Gamma(TTM_{0})$. Then, we define the covariant Chern derivative of  $T$ in the  direction of $X$ by the formula
 \begin{align}\label{CovaDer1}
   (\nabla_{X}T)&(\overline{Y}_{1}, \ldots ,\overline{Y}_{p}, X_{1},\ldots,X_{q}, \alpha_{1},\ldots, \alpha_{r}) \nonumber\\
   &=
  X( T(\overline{Y}_{1},\ldots,\overline{Y}_{p}, X_{1},...,X_{q},\alpha_{1},\ldots, \alpha_{r}))\nonumber\\
  &- \sum_{i=1}^{p}(\overline{Y}_{1},\ldots,\nabla_{X}\overline{Y}_{i},\ldots,\overline{Y}_{p}, X_{1},\ldots,X_{q},
 \alpha_{1},\ldots, \alpha_{r})\nonumber\\
 &- \sum_{j=1}^{q} T(\overline{Y}_{1},\ldots,\overline{Y}_{p}, X_{1},\ldots,(\nabla_{X}\pi_{\ast}X_{j})^{H},\ldots,X_{q},
 \alpha_{1},\ldots, \alpha_{r})\nonumber\\
 &- \sum_{j=1}^{q}T(\overline{Y}_{1},\ldots,\overline{Y}_{p}, X_{1},\ldots,(\nabla_{X}\theta X_{j})^{V},\ldots,X_{q}, \alpha_{1},\ldots,
 \alpha_{r})\nonumber\\
& - \sum_{k=1}^{r}T(\overline{Y}_{1},\ldots,\overline{Y}_{p}, X_{1},\ldots,X_{q}, \alpha_{1},\ldots,\nabla_{X}\alpha_{k},\ldots, \alpha_{r}),
 \end{align}
where $X_{i}\in \Gamma(TTM_{0}),\: i=1,\ldots,p$; $\overline{Y}_{j} \in \Gamma(\pi^{\ast}TM), \: j= 1,\ldots,q$; $\alpha_{k} \in \Gamma (\pi^{\ast}T^{\ast}M),\: k= 1,\ldots,r$, and each of the quantities $\nabla_{X}\alpha_{k}$ is evaluated by
$
(\nabla_{X}\alpha_{k})(\overline{Y} )= X \alpha_{k}(\overline{Y})- \alpha_{k}(\nabla_{X}\overline{Y}).
$
}
\end{definition}

 \section{Almost contact Finsler structures}\label{Almost}

 In this section, we adapt the definition of almost contact structures given in \cite{bl} in the case of Finsler.

 Let $\phi$, $\xi$ and $\eta$ be the  $(1,0;1)$-, $(1,0;0)$- and $(0,0;1)$-Finslerians tensor,  respectively, such that
 \begin{equation}  \label{Ac1}
 \phi^{2}   =  - \mathbb{I} +\eta \otimes \xi \;\;\mbox{and}\;\;
 \eta (\xi)   =  1.
 \end{equation} 
Then, the triplet $(\phi, \eta, \xi)$  is called an almost contact Finsler structure on $\pi^{\ast}TM$ and $(\pi^{\ast}TM, \phi, \eta, \xi)$ is called almost contact Finsler pulled-back bundle.

 First of all, we prove the following.

 \begin{proposition}\label{ExtAlmo}
 Let  $(\pi^{\ast}TM, \phi, \eta, \xi)$ be an almost contact Finsler pulled-back bundle. Then,
 $$
 \phi(\xi)= 0\;\;\;\mbox{and}\;\;\;\eta \circ \phi = 0.
 $$
 Moreover, $\phi$ is of rank $2n$.
 \end{proposition}
 \begin{proof}
 The assertion (i) is as follows. By definition of Almost contact Finsler structure, i.e., the relations (\ref{Ac1}) , we have
$
 \phi^{2} \xi = -\xi+ \eta(\xi)\xi = 0.
$
 Then, $\phi(\xi)= 0$ or $\phi(\xi)$ is a non trivial eigenvector of $\phi$ corresponding to eigenvalue $0$.  By (\ref{Ac1}) again, we have
 \begin{equation}\label{r1}
 \phi^{2}\phi(\xi)=0  \Longleftrightarrow    \phi(\xi)=\eta(\phi(\xi))\xi.
 \end{equation}
 As $\phi(\xi)$ is a non trivial eigenvector of $\phi$ corresponding to the eigenvalue $0$, we have $\eta(\phi (\xi))\neq 0$.
 Thus by (\ref{r1}) we have
$
 0=\phi(\phi(\xi))=\eta(\phi(\xi))\phi(\xi)=\left(\eta(\phi(\xi))\right)^{2}\xi \neq 0, 
$
 which is a contradiction. Therefore,
 $
 \phi(\xi)=0.
 $
Now let us prove (ii).  Using (\ref{Ac1}), we observe that, for all $\overline{X}\in \Gamma(\pi^{\ast}TM)$,
$
 \phi^{2}(\overline{X})=-\overline{X} + \eta(\overline{X}) \xi. 
$
 Applying $\phi$ to this equation, one obtains
 \begin{equation}\label{rp3}
 \phi^{3}(\overline{X})= -\phi(\overline{X}).
 \end{equation}
From (\ref{Ac1}) and using (\ref{rp3}) , we have
$
 \eta(\phi(\overline{X})) \xi = \phi^{3}(\overline{X})+ \phi(\overline{X})=0.
$
Hence $\eta \circ \phi = 0$. The last assertion is proven as follows. For all $(x,y) \in TM_{0}$, $ (\pi^{\ast}TM)\lvert_{(x,y)}$ is of odd dimension, i.e., $2n+1$.  Therefore, it is sufficient to show that $\ker \phi = < \xi >$. Since  $\phi(\xi)=0$, we have $ < \xi > \subseteq \ker \phi$. Now let $\overline{\xi} \in \ker \phi$. Then  $\phi(\overline{\xi})= 0$, and using the relation (\ref{Ac1}), one gets
 $
 \overline{\xi}= \eta(\overline{\xi})\xi,
$
 It follows that, $\ker \phi\subseteq < \xi >$. Thus
 $
\ker \phi= < \xi >,
 $
 which completes the proof.
 \end{proof}
Note that the Proposition \ref{ExtAlmo} is a Finsler pulled-back bundle extension to the one of the tangent bundle given in \cite{bl}.

Now, we define the normality condition on an almost contact Finsler structures on  $(\pi^{\ast}TM, \phi, \eta, \xi)$. Let $T$ be a $(1,0;1)$-Finslerian tensor. The Nijenhuis torsion of $T$, is the $(2,0;1)$-Finslerian tensor decomposed into the horizontal and vertical part $N_{T}^{H}$ and $N_{T}^{V}$ given, respectively,  by
 \begin{eqnarray}
 N_{T}^{H}(\overline{X}, \overline{Y}) &=& T^{2}\pi_{\ast}[\overline{X}^{H}, \overline{Y}^{H}] + \pi_{\ast}[(T\overline{X})^{H}, (T\overline{Y})^{H}]
 - T \pi_{\ast}[(T\overline{X})^{H}, \overline{Y}^{H}]\cr
& &- T\pi_{\ast}[\overline{X}^{H}, (T\overline{Y})^{H}],
 \end{eqnarray}
 and
 \begin{eqnarray}
 N_{T}^{V}(\overline{X}, \overline{Y}) &=& T^{2}\theta[\overline{X}^{V}, \overline{Y}^{V}] + \theta[(T\overline{X})^{V}, (T\overline{Y})^{V}]
 - T \theta[(T\overline{X})^{V}, \overline{Y}^{V}]\cr
& &- T\theta[\overline{X}^{V}, (T\overline{Y})^{V}],
 \end{eqnarray}
 for any $\overline{X}, \overline{Y} \in \Gamma(\pi^{\ast}TM)$.

  \begin{definition}{\rm
 The almost contact Finsler structure $(\phi, \eta, \xi)$ on pulled-back bundle $\pi^{\ast}TM$ is horizontally normal if
 \begin{equation}
 \mathcal{N}_{H}^{(1)}(\overline{X}, \overline{Y})= N_{\phi}^{H}(\overline{X}, \overline{Y}) +2 d^{H}\eta(\overline{X}, \overline{Y})\xi = 0,
 \end{equation}
 and it is vertically normal if
 \begin{equation}
 \mathcal{N}_{V}^{(1)}(\overline{X}, \overline{Y})= N_{\phi}^{V}(\overline{X}, \overline{Y}) +2 d^{V}\eta(\overline{X}, \overline{Y})\xi = 0,
 \end{equation}
for any $\overline{X}, \overline{Y} \in \Gamma(\pi^{\ast}TM)$.
 }
 \end{definition}
Next, we give some equivalent conditions for horizontal and vertical normality of the structure $(\phi, \eta, \xi)$. For this reason, we introduce
 six tensors $\mathcal{N}_{H}^{(2)} $,  $\mathcal{N}_{V}^{(2)} $,  $\mathcal{N}_{H}^{(3)} $,  $\mathcal{N}_{V}^{(3)}$, $ \mathcal{N}_{H}^{(4)}$ and
 $\mathcal{N}_{V}^{(4)}$ given by
 \begin{eqnarray}
  \mathcal{N}_{H}^{(2)}(\overline{X}, \overline{Y})&:=& (\mathcal{L}_{\phi\overline{X}}^{H}\eta)\overline{Y}-
  (\mathcal{L}_{\phi\overline{Y}}^{H}
  \eta)\overline{X},\\
  \mathcal{N}_{V}^{(2)}(\overline{X}, \overline{Y})&:=& (\mathcal{L}_{\phi\overline{X}}^{V}\eta)\overline{Y}-(\mathcal{L}_{\phi\overline{Y}}^{V}
  \eta)\overline{X},\\
  \mathcal{N}_{H}^{(3)} &:=&   (\mathcal{L}_{\xi}^{H}\phi)\overline{X},\\
   \mathcal{N}_{V}^{(3)} &:=&   (\mathcal{L}_{\xi}^{V}\phi)\overline{X},\\
   \mathcal{N}_{H}^{(4)} &:=&   (\mathcal{L}_{\xi}^{H}\eta)\overline{X},\\
   \mathcal{N}_{V}^{(4)} &:=&  (\mathcal{L}_{\xi}^{V}\eta)\overline{X},
 \end{eqnarray}
 for all  $\overline{X}, \overline{Y} \in \Gamma(\pi^{\ast}TM)$.
 \begin{theorem}
 For an almost contact Finsler structure $(\phi, \eta, \xi)$, the vanishing of  $\mathcal{N}_{H}^{(1)}$, implies the vanishing of
  $\mathcal{N}_{H}^{(2)}$, $ \mathcal{N}_{H}^{(3)}$ and $ \mathcal{N}_{H}^{(4)}$. Likewise, the vanishing of $ \mathcal{N}_{V}^{(1)}$
  implies the vanishing of $ \mathcal{N}_{V}^{(2)}$,$ \mathcal{N}_{V}^{(3)}$ and $ \mathcal{N}_{V}^{(4)}$.
 \end{theorem}
 \begin{proof}
  The proof is similar to the one given in Riemannian case, by Blair in \cite{bl}.
 \end{proof}
Now, let $(\phi,\eta, \xi)$  an almost contact Finsler structure on $\pi^{\ast}TM$. When, the  fundamental tensor $g$ of the Finslerian
structure $F$, satisfy
 \begin{equation}
 g(\phi \overline{X}, \phi \overline{Y}) = g(\overline{X}, \overline{Y}) -\eta(\overline{X})\eta(\overline{Y}),
 \end{equation}
 for any sections $\overline{X}$ and $\overline{Y}$ on $\pi^{\ast}TM$, we said that $g$ is compatible with the structure $(\phi,\eta, \xi)$. In this case, $(\pi^{\ast}TM,\phi,\eta, \xi,g)$ is called an almost
 contact metric Finsler pulled-back bundle.

 Moreover, we define the generalized second fundamental $2$-form $\Phi$ by:
 \begin{equation}
 \Phi(\overline{X},\overline{Y}):= g(\overline{X},\phi\overline{Y}),\quad\quad \overline{X},\overline{Y} \in \Gamma(\pi^{\ast}TM).
 \end{equation}
 Because of the isomorphisms in (\ref{Isomorphs}), we defined an almost contact metric Finsler structure as follows.
 \begin{definition}{\rm
 An almost contact metric Finsler structure  $(\phi,\eta, \xi,g)$ is called \textit{contact Finsler structure} if
 \begin{equation}\label{ct1}
   \Phi=2d^{H}\eta, \;\;\mbox{or}\;\;\Phi = 2d^{V}\eta.
 \end{equation}  
 }
 \end{definition} 
 Let $(\phi,\eta, \xi,g)$ be a contact metric Finsler structure on $\pi^{\ast}TM$. A section $\overline{X} \in \Gamma(\pi^{\ast}TM)  $,
 is horizontally  Killing  and vertically Killing, if it satisfies, respectively
 \begin{equation}
 \mathcal{L}^{H}_{\overline{X}}g = 0\;\;\mbox{and}\;\; \mathcal{L}^{V}_{\overline{X}}g = 0,
 \end{equation}
 respectively.
 
If $\xi$ is  horizontally  Killing (resp. vertically Killing), then $(\phi,\eta, \xi,g)$ is called horizontal $K$-contact Finsler structure (resp. vertical $K$-contact Finsler structure). 
 \begin{definition}{\rm
Let $(\phi,\eta, \xi,g)$ be a contact metric Finsler structure. If $\xi$ is  horizontally  Killing (resp. vertically Killing), then $(\phi,\eta, \xi,g)$ is called horizontal $K$-contact Finsler structure (resp. vertical $K$-contact Finsler structure).
}
 \end{definition}
 \begin{theorem}\label{theofdci}
 Let $(\phi,\eta, \xi,g)$ be a contact metric Finsler structure on $\pi^{\ast}TM$. Then
 \begin{equation}
 \mathcal{N}_{H}^{(4)}=\mathcal{N}_{V}^{(4)} = 0,\quad\quad \mathcal{N}_{H}^{(2)}=\mathcal{N}_{V}^{(2)}=0.
 \end{equation}
 Moreover
 $\mathcal{N}_{H}^{(3)}$ vanish if and only if $\xi$ is horizontally Killing and $\mathcal{N}_{V}^{(3)}$
 vanish if and only if $\xi$ is vertically Killing.
 \end{theorem}
 \begin{proof}
 The proof is similar to the one for the Riemannian case in \cite{bl}.
 \end{proof}
 \begin{lemma}\label{lemfdci}
 For an almost contact metric Finsler structure $(\phi,\eta,\xi,g)$ on $\pi^{\ast}TM$ with identification $\pi^{\ast}TM\cong \mathcal{H}T M_{0}$, the covariant derivative of $\phi$
 with respect to the Chern connection is given by
 \begin{align} \label{dcfi}
  2g((\nabla_{X}\phi)\overline{Y}, \overline{Z})&=  3 d^{H}\Phi(\pi_{\ast}X, \phi\overline{Y},\phi\overline{Z})-
  3d^{H}\Phi(\pi_{\ast}X,\overline{Y},\overline{Z})\nonumber\\
  &+  g(\mathcal{N}_{H}^{(1)}(\overline{Y},\overline{Z}),\phi \pi_{\ast}X)
  + \mathcal{N}_{H}^{(2)}(\overline{Y},\overline{Z})\eta(\pi_{\ast}X)\nonumber\\
  & +  2 d^{H}\eta(\phi \overline{Y}, \pi_{\ast}X)\eta(\overline{Z})- 2 d^{H}\eta(\phi \overline{Z}, \pi_{\ast}X)\eta(\overline{Y})\cr
  &-  2A(\theta(X), \phi\overline{Y}, \overline{Z})- 2A(\theta(X), \overline{Y}, \phi \overline{Z}),
 \end{align}
where $X\in\Gamma(TTM_{0})$ and $\overline{Y},\overline{Z} \in \Gamma(\pi^{\ast}TM)$.
\end{lemma}
\begin{proof}
For $X,Y,Z \in \Gamma(TTM_{0})$, it is well know, by the Chern connection, that \cite{mb1}
 \begin{align}\label{et4}
   2g(\nabla_{X}\pi_{\ast}Y, \pi_{\ast}Z)&=   X.g(\pi_{\ast}Y,\pi_{\ast}Z) + Y.g(\pi_{\ast}Z,\pi_{\ast}X)
   - Z.g(\pi_{\ast}X,\pi_{\ast}Y)\cr
   & + g(\pi_{\ast}[X,Y], \pi_{\ast}Z) -g(\pi_{\ast}[Y,Z], \pi_{\ast}X)+ g(\pi_{\ast}[Z,X], \pi_{\ast}Y)\nonumber\\
   &  -2\mathcal{A}(X,Y,Z),
\end{align}
 where
\begin{equation}
 \mathcal{A}(X,Y,Z)= A(\theta(X),\pi_{\ast}Y,\pi_{\ast}Z))+A(\theta(Y),\pi_{\ast}Z,\pi_{\ast}X))- A(\theta(Z),\pi_{\ast}X,\pi_{\ast}Y)).
\end{equation}
Now,
\begin{align}\label{mdfi}
& 2g((\nabla_{X}\phi)\overline{Y}, \overline{Z})  = 2g(\nabla_{X}\phi(\overline{Y}),\overline{Z})
 + 2 g(\nabla_{X}\overline{Y},\phi(\overline{Z}))\nonumber\\ 
 &= X \Phi(\overline{Z},\overline{Y}) + \phi\overline{Y}^{H}\left(\Phi(\overline{Z}, \pi_{\ast}X)+
 \eta(\overline{Z})\eta(\pi_{\ast}X)\right)- \overline{Z}^{H}\Phi(\pi_{\ast}X, \overline{Y})\nonumber\\
 &- \Phi(\pi_{\ast}[X,\phi\overline{Y}^{H}], \phi \overline{Z})+ \eta(\pi_{\ast}[X,\phi\overline{Y}^{H}])\eta(\overline{Z})
 - g(\phi\pi_{\ast}[\phi \overline{Y}^{H}, \overline{Z}^{H}],  \phi\pi_{\ast}X)\nonumber\\
 &- \eta(\pi_{\ast}[\phi \overline{Y}^{H}, \overline{Z}^{H}])\eta(\pi_{\ast}X)+ \Phi(\pi_{\ast}[\overline{Z}^{H},X], \overline{Y})
 -2A(\theta(X), \phi \overline{Y}, \overline{Z})\nonumber\\
 &+ X \Phi(\phi\overline{Y}, \phi\overline{Z})- \overline{Y}^{H}\Phi(\overline{Z},\pi_{\ast}X)-
 \phi \overline{Z}^{H}\left(\Phi(\phi \overline{Y}, \pi_{\ast}X)+ \eta(\pi_{\ast}X, \eta(\overline{Y})\right)\nonumber\\
 &+ \Phi(\pi_{\ast}[X,\overline{Y}^{H}],\overline{Z}^{H})- g(\phi\pi_{\ast}[\overline{Y}^{H}, \phi\overline{Z}^{H}], \phi\pi_{\ast}X)
 - \eta(\pi_{\ast}[\overline{Y}^{H}, \phi\overline{Z}^{H}])\eta(\pi_{\ast}X)\nonumber\\
 &- \Phi( \pi_{\ast}[\phi\overline{Z}^{H}, X], \phi \overline{Y}) + \eta(\pi_{\ast}[\phi\overline{Z}^{H}, X])\eta(\overline{Y})
 -2 A(\theta(X), \overline{Y},\phi\overline{Z})\nonumber\\&+
 \Phi(\pi_{\ast}[\overline{Y}^{H},\overline{Z}^{H}], \pi_{\ast}X)-
 g((\pi_{\ast}[\overline{Y}^{H},\overline{Z}^{H}], \phi\pi_{\ast}X)) - \Phi(\pi_{\ast}[\phi\overline{Y}^{H},\phi\overline{Z}^{H}], \pi_{\ast}X)\nonumber\\
 &+g((\pi_{\ast}[\phi\overline{Y}^{H},\phi\overline{Z}^{H}], \phi\pi_{\ast}X))+ g(2d^{H}\eta(\overline{Y},\overline{Z})\xi, \phi \pi_{\ast}X).
\end{align}
Noted that, the sum of  the last five terms of (\ref{mdfi}) is zero, and is maintained to write this expression in terms of horizontal exterior
differential of $\Phi$, and the tensors $\mathcal{N}_{H}^{(1)}$ and $\mathcal{N}_{H}^{(2)}$.
It follows that 
\begin{align}
 & 2g((\nabla_{X}\phi)\overline{Y}, \overline{Z})  = X \Phi(\phi\overline{Y},\phi\overline{Z}) + \phi\overline{Y}^{H}\Phi(\phi\overline{Z}, \pi_{\ast}X)
 +\phi\overline{Z}^{H}\Phi( \pi_{\ast}X,\phi\overline{Y})\nonumber \\
 & - \Phi(\pi_{\ast}[X,\phi\overline{Y}^{H}], \phi \overline{Z}) - \Phi(\pi_{\ast}[\phi\overline{Y}^{H},\phi\overline{Z}^{H}], \pi_{\ast}X)
 - \Phi(\pi_{\ast}[\phi\overline{Z}^{H},X], \phi \overline{Y})\nonumber \\
 &-X \Phi(\overline{Y},\overline{Z})-\overline{Y}^{H}\Phi(\overline{Z}, \pi_{\ast}X)-
 \overline{Z}^{H}\Phi( \pi_{\ast}X,\overline{Y})+ \Phi(\pi_{\ast}[X,\overline{Y}^{H}],\overline{Z})\nonumber \\
 & + \Phi(\pi_{\ast}[\overline{Y}^{H},\overline{Z}^{H}],\pi_{\ast}X)+ \Phi(\pi_{\ast}[\overline{Z}^{H}, X], \overline{Y})+ 
 \phi\overline{Y}^{H}\eta(\overline{Z})\eta(\pi_{\ast}X)\nonumber \\
 &- \eta(\pi_{\ast}[\phi\overline{Y}^{H}, \overline{Z}^{H}])\eta(\pi_{\ast}X)+\phi\overline{Z}^{H}\eta(\overline{Y})\eta(\pi_{\ast}X)
 - \eta(\pi_{\ast}[\overline{Y}^{H}, \phi\overline{Z}^{H}])\eta(\pi_{\ast}X)\nonumber \\
 &- g(\phi\pi_{\ast}[\phi\overline{Y}^{H}, \overline{Z}^{H}], \phi \pi_{\ast}X)
 -g(\phi\pi_{\ast}[\overline{Y}^{H}, \phi\overline{Z}^{H}], \phi \pi_{\ast}X)- g(\pi_{\ast}[\overline{Y}^{H},\overline{Z}^{H}], \phi\pi_{\ast}X)
 \nonumber \\
 &+ g(\pi_{\ast}[\phi\overline{Y}^{H},\phi\overline{Z}^{H}], \phi \pi_{\ast}X) + g(2d^{H}\eta(\overline{Y}, \overline{Z})\xi, \phi\pi_{\ast}X)
 + \eta(\pi_{\ast}[X, \phi \overline{Y}^{H}])\eta(\overline{Z})\nonumber \\
 & + \eta(\pi_{\ast}[\phi\overline{Z}^{H}, X])\eta(\overline{Y})
 -2(A(\theta(X), \phi \overline{Y}, \overline{Z})+ A(\theta(X), \overline{Y}, \phi \overline{Z}))
 \nonumber \\
 &=  3 d^{H}\Phi(\pi_{\ast}X, \phi\overline{Y},\phi\overline{Z})-
  3d^{H}\Phi(\pi_{\ast}X,\overline{Y},\overline{Z})+ \mathcal{N}_{H}^{(2)}(\overline{Y},\overline{Z})\eta(\pi_{\ast}X)\nonumber\\
  &+  g(\mathcal{N}_{H}^{(1)}(\overline{Y},\overline{Z}),\phi \pi_{\ast}X)   
   +  2 d^{H}\eta(\phi \overline{Y}, \pi_{\ast}X)\eta(\overline{Z})- 2 d^{H}\eta(\phi \overline{Z}, \pi_{\ast}X)\eta(\overline{Y})
 \nonumber \\&   -2(A(\theta(X), \phi \overline{Y}, \overline{Z})+ A(\theta(X), \overline{Y}, \phi \overline{Z})).
\end{align}
Which completes the proof.
\end{proof}  
The relation (\ref{dcfi}) generalizes the one given by Blair in \cite[page 82]{bl} for the case of almost contact metric structure $(\phi, \xi, \eta, g)$. Indeed, when the Finsler structure $F$ is Riemannian, the Cartan tensor vanishes and the Chern Connection reduces to the Levi-Civita connection of $g$.  
 \begin{corollary}
  For an almost contact metric Finsler structure $(\phi, \eta, \xi, g)$ on $\pi^{\ast}TM$ with identification $\pi^{\ast}TM\cong \mathcal{H}T M_{0}$, the horizontal and vertical covariant derivative of $\phi$, are given
  respectively by:
  \begin{align}
   2g((\nabla^{H}_{X}\phi)\overline{Y}, \overline{Z})&= 3 d^{H}\Phi(\pi_{\ast}X, \phi\overline{Y},\phi\overline{Z})-
  3d^{H}\Phi(\pi_{\ast}X,\overline{Y},\overline{Z})\nonumber\\
  &+  g(\mathcal{N}_{H}^{(1)}(\overline{Y},\overline{Z}),\phi \pi_{\ast}X) + \mathcal{N}_{H}^{(2)}(\overline{Y},\overline{Z})\eta(\pi_{\ast}X)\nonumber\\
  & +  2 d^{H}\eta(\phi \overline{Y}, \pi_{\ast}X)\eta(\overline{Z})
  - 2 d^{H}\eta(\phi \overline{Z}, \pi_{\ast}X)\eta(\overline{Y})
  \end{align}
and
\begin{equation}
 2g((\nabla^{V}_{X}\phi)\overline{Y}, \overline{Z})=
  - 2 \left(A(\theta(X), \phi\overline{Y}, \overline{Z})+A(\theta(X), \overline{Y}, \phi \overline{Z})\right),
\end{equation}
where $\nabla^{H}_{X}= \nabla_{X^{H}}$ and $ \nabla^{V}_{X}= \nabla_{X^{V}}$.
 \end{corollary}
\begin{proof}
 The proof follows from a straightforward calculation using Lemma \ref{lemfdci}.
\end{proof}
  \begin{definition}{\rm
 The horizontal normal contact Finsler structures $(\phi,\eta, \xi,g)$ are called \textit{horizontal Sasakian Finsler structures}, and  the vertical
 ones are called \textit{vertical Sasakian Finsler structures}.}
 \end{definition}
\begin{theorem}\label{dcphi}
 An almost contact metric Finsler structure $(\phi, \eta, \xi, g)$ on $\pi^{\ast}TM$, with identification $\pi^{\ast}TM\cong \mathcal{H}T M_{0}$, is horizontally Sasakian if and only if 
 \begin{equation}
   (\nabla_{X}\phi)\overline{Y}= g(\pi_{\ast}X,\overline{Y})\xi-\eta(\overline{Y})\pi_{\ast}X + 
  \phi A^{\sharp}(\theta(X), \overline{Y}, \bullet)-A^{\sharp}(\theta(X), \phi \overline{Y}, \bullet),\nonumber 
 \end{equation}
 where
\begin{equation}
g( A^{\sharp}(\theta(X), \overline{Y}, \bullet), \overline{Z}) = A(\theta(X), \overline{Y},\overline{Z}), \nonumber
\end{equation}

for all $X\in \Gamma(TTM_{0})$ and $\overline{Y}, \overline{Z} \in \Gamma(\pi^{\ast}TM)$.
\end{theorem}
\begin{proof}
 Suppose that $(\phi, \eta, \xi, g)$ is horizontally Sasakian, then by Lemma \ref{lemfdci},  
 \begin{align} 
  2g((\nabla_{X}\phi)\overline{Y}, \overline{Z})&=
    2 d^{H}\eta(\phi \overline{Y}, \pi_{\ast}X)\eta(\overline{Z})- 2 d^{H}\eta(\phi \overline{Z}, \pi_{\ast}X)\eta(\overline{Y})\cr
  &-  2A(\theta(X), \phi\overline{Y}, \overline{Z})- 2A(\theta(X), \overline{Y}, \phi \overline{Z})\cr
  &= 2g( \overline{Y}, \pi_{\ast}X)\eta(\overline{Z})-2 g(\pi_{\ast}X, \overline{Z})\eta(\overline{Y})
 \cr & -  2g(A^{\sharp}(\theta(X), \phi\overline{Y}, \bullet), \overline{Z})- 2g(A^{\sharp}(\theta(X), \overline{Y}, \bullet), \phi\overline{Z})\cr
 &= 2g(g(\overline{Y}, \pi_{\ast}X)\xi-\eta(\overline{Y})\pi_{\ast}X
 + \phi A^{\sharp}(\theta(X), \overline{Y}, \bullet),\overline{Z})\cr
 &-2g(A^{\sharp}(\theta(X), \phi\overline{Y}, \bullet), \overline{Z}).\nonumber
 \end{align}
 From which the result follows. Conversely assuming that 
 \begin{equation}
   (\nabla_{X}\phi)\overline{Y}= g(\pi_{\ast}X,\overline{Y})\xi-\eta(\overline{Y})\pi_{\ast}X + 
  \phi A^{\sharp}(\theta(X), \overline{Y}, \bullet)-A^{\sharp}(\theta(X), \phi \overline{Y}, \bullet),\nonumber 
 \end{equation}
 and taking $\overline{Y}= \xi$, we have 
 \begin{equation}
  (\nabla_{X}\phi)\xi = \eta(\pi_{\ast}X)\xi-\pi_{\ast}X + \phi A^{\sharp}(\theta(X), \xi,\bullet). \nonumber
 \end{equation}
Hence
\begin{equation}
 \nabla_{X}\xi= -\phi\pi_{\ast}X + \phi^{2} A^{\sharp}(\theta(X), \xi, \bullet).\nonumber
\end{equation}
Therefore
\begin{align}
 d^{H}\eta(\overline{X}, \overline{Y})&= \frac{1}{2} ( g(\overline{Y}, \nabla_{\overline{X}^{H}}\xi) - g(\overline{X}, \nabla_{\overline{Y}^{H}}\xi))\nonumber\\
 &= g(\overline{X}, \phi\overline{Y}) = \Phi(\overline{X}, \overline{Y}). \nonumber
\end{align}
Thus $(\phi, \xi, \eta, g)$ is a contact metric Finsler structure.
Now, by definition,
\begin{align}
 N^{H}_{\phi}(\overline{X}, \overline{Y})= \phi(\nabla_{\overline{Y}^{H}}\phi)\overline{X}- \phi(\nabla_{\overline{X}^{H}}\phi)\overline{Y}
 + (\nabla_{(\phi \overline{X})^{H}}\phi)\overline{Y}-(\nabla_{(\phi \overline{Y})^{H}}\phi)\overline{X},\nonumber
\end{align}
and using the hypothesis, we obtain
\begin{align}
  N^{H}_{\phi}(\overline{X}, \overline{Y})= -2 d^{H}\eta(\overline{X}, \overline{Y})\xi,\nonumber
\end{align}
which completes the proof.
\end{proof}
\begin{corollary}
 For an almost contact metric Finsler structure $(\phi, \eta, \xi, g)$ on $\pi^{\ast}TM$, with identification $\pi^{\ast}TM\cong \mathcal{H}T M_{0}$, the following holds:
 \begin{equation}
\nabla_{\xi^{H}}\phi = 0\;\;\mbox{and}\;\; (\nabla_{\xi^{V}}\phi)\overline{Y} = \phi A^{\sharp}(\theta(\xi^{V}), \overline{Y},
 \bullet)-A^{\sharp}(\theta(\xi^{V}), \phi \overline{Y},\bullet),\nonumber
 \end{equation} 
 for all $\overline{Y} \in \Gamma(\pi^{\ast}TM)$.
\end{corollary}
\begin{lemma}\label{cordcx}
 For a contact metric Finsler structure $(\phi,\eta,\xi, g)$ on $\pi^{\ast}TM$, with identification $\pi^{\ast}TM\cong \mathcal{H}T M_{0}$, the following holds:
 \begin{align} \label{corf}
  &2g((\nabla_{X}\phi)\overline{Y}, \overline{Z}) =   g(\mathcal{N}_{H}^{(1)}(\overline{Y},\overline{Z}),\phi \pi_{\ast}X)
 + 2 d^{H}\eta(\phi \overline{Y}, \pi_{\ast}X)\eta(\overline{Z})\nonumber\\
 &-  2 d^{H}\eta(\phi \overline{Z}, \pi_{\ast}X)\eta(\overline{Y})
 - 2\left(A(\theta(X), \phi\overline{Y}, \overline{Z})+A(\theta(X), \overline{Y}, \phi \overline{Z})\right).
 \end{align}
 \end{lemma}
\begin{proof}
 The proof is  obtained by Lemma \ref{lemfdci} and Theorem \ref{theofdci}.
\end{proof}
On a contact metric Finsler structure (Theorem \ref{theofdci}), $\mathcal{N}^{(3)}_{H}=0$ if and only if
$\xi$ is horizontally Killing, and $\mathcal{N}^{(3)}_{V}=0$ if and only if $\xi$ is vertically Killing. Then the horizontal Sasakian Finsler
structure is horizontal K-contact structure and the vertical Sasakian Finsler structure is vertical K-contact structure.

Due to the use of tensors $\mathcal{N}_{H}^{(3)}$ and $\mathcal{N}_{V}^{(3)}$, it would be important to investigate some properties of these tensors. 

Let
\begin{eqnarray}
 \mathbf{h}= \frac{1}{2}\mathcal{N}_{H}^{(3)}= \frac{1}{2}\mathcal{L}_{\xi^{H}}\phi \quad \hbox{ and} \quad \mathbf{v}= \frac{1}{2} \mathcal{N}_{V}^{(3)}
=\frac{1}{2}\mathcal{L}_{\xi^{V}}\phi.
\end{eqnarray}
We firstly notice that 
\begin{equation}
 \mathbf{h}\xi= 0 \quad \hbox{and} \quad \mathbf{v}\xi= 0.
\end{equation}
\begin{proposition}
For a contact metric Finsler $(\phi, \eta, \xi, g)$, $\mathbf{h}$ is a symmetric operator, whereas $\mathbf{v}$ is a symmetric operator
if
\begin{equation}
    A(\theta(\overline{X}^{V}\!), \xi, \phi\overline{Y}) +  A(\theta(\phi\overline{X}^{V}\!), \xi,\overline{Y}) =
    A(\theta(\overline{Y}^{V}\!), \xi, \phi\overline{X}) + A(\theta(\phi\overline{Y}^{V}\!), \xi, \overline{X}),\nonumber
   \end{equation}
 for all $\overline{X},\overline{Y} \in \Gamma(\pi^{\ast}TM)$.
\end{proposition}
\begin{proof}
 Let $\overline{X},\overline{Y} \in \Gamma(\pi^{\ast}TM)$, we have
 \begin{align}
  2g(\mathbf{h}\overline{X}, \overline{Y})&= g((\mathcal{L}_{\xi^{H}}\phi)\overline{X},\overline{Y})=  g(\pi_{\ast}[\xi^{H}, (\phi\overline{X})^{H}]-\phi(\pi_{\ast}[\xi^{H}, \overline{X}^{H}]), \overline{Y})\cr
  &=  g(-\nabla_{(\phi\overline{X})^{H}}\xi+ \phi(\nabla_{\overline{X}^{H}}\xi), \overline{Y}).\nonumber
 \end{align}
It follows that if $\overline{X}$ or $\overline{Y}$ is equal to $\xi$ then $g(\mathbf{h}\overline{X}, \overline{Y})=0$. Recall that $\mathcal{N}_{H}^{(2)}= 0$ for contact metric Finsler structures. Then for $\overline{X}$ and $\overline{Y}$ orthogonal to $\xi$,
\begin{eqnarray}
 0&=&\mathcal{N}_{H}^{(2)}=(\mathcal{L}_{(\phi \overline{X})^{H}}\eta)\overline{Y}-(\mathcal{L}_{(\phi \overline{Y})^{H}}\eta)\overline{X}\cr
 &=& \eta(\pi_{\ast}[(\phi \overline{Y})^{H}, \overline{X}^{H}])- \eta(\pi_{\ast}[(\phi\overline{X})^{H},\overline{Y}^{H}]).
\end{eqnarray}
Thus
\begin{eqnarray}\label{dcx}
 \eta(\nabla_{(\phi\overline{X})^{H}}\overline{Y})+ \eta(\nabla_{\overline{X}^{H}}\phi \overline{Y})=
 \eta(\nabla_{(\phi\overline{Y})^{H}}\overline{X})+ \eta(\nabla_{\overline{Y}^{H}}\phi \overline{X}).
\end{eqnarray}
On the other hand, $g(\xi, \overline{Y})= 0$ implies that
$
  g(\xi, \nabla_{(\phi\overline{X})^{H}}\overline{Y})= - g(\nabla_{(\phi \overline{X})^{H}}\xi,
 \overline{Y}), 
 $
so we have
\begin{align}
  2g(\mathbf{h}\overline{X}, \overline{Y})&= g(-\nabla_{(\phi\overline{X})^{H}}\xi+ \phi(\nabla_{\overline{X}^{H}}\xi), \overline{Y})\cr 
  &=  \eta(\nabla_{(\phi\overline{X})^{H}}\overline{Y})+ \eta(\nabla_{\overline{X}^{H}}\phi \overline{Y})\cr
  &= \eta(\nabla_{(\phi\overline{Y})^{H}}\overline{X})+ \eta(\nabla_{\overline{Y}^{H}}\phi \overline{X})\cr
&=    2g(\mathbf{h}\overline{Y}, \overline{X}),\nonumber
\end{align}
which proves that $\mathbf{h}$ is symmetric. For the operator $\mathbf{v}$, using the fact that (see \cite{AE} for details)
\begin{equation}
 \nabla_{X}\theta(Y)-\nabla_{Y}\theta(X)= \theta([X,Y]), \nonumber
\end{equation}
we obtain
\begin{eqnarray}\label{dcxv}
 \eta(\nabla_{(\phi\overline{X})^{V}}\overline{Y})+ \eta(\nabla_{\overline{X}^{V}}\phi \overline{Y})=
 \eta(\nabla_{(\phi\overline{Y})^{V}}\overline{X})+ \eta(\nabla_{\overline{Y}^{V}}\phi \overline{X}).
\end{eqnarray}
By observing that
\begin{equation}
 g(\xi, \nabla_{\overline{X}^{V}}\phi \overline{Y})+ g(\nabla_{\overline{X}^{V}}\xi, \phi\overline{Y})
 = 2A(\theta(\overline{X}^{V}), \xi, \phi \overline{Y})),
\end{equation}
one gets
\begin{align}
 &2g(\mathbf{v}\overline{X},\overline{Y}) =  g(\xi, \nabla_{(\pi\overline{X})^{V}}\overline{Y})+
 g(\xi, \nabla_{\overline{X}^{V}}\phi\overline{Y}) -2 A((\phi\overline{X})^{V},\xi, \overline{Y})\cr
 &  -2 A(\theta(\overline{X}^{V}), \xi, \phi\overline{Y})\nonumber\\
 &=  \eta(\nabla_{(\pi\overline{X})^{V}}\overline{Y})+ \eta(\nabla_{\overline{X}^{V}}\phi\overline{Y})
 -2 (A((\phi\overline{X})^{V},\xi, \overline{Y})+A(\theta(\overline{X}^{V}), \xi, \phi\overline{Y}) )\cr
 &=
  \eta(\nabla_{(\pi\overline{Y})^{V}}\overline{X})+ \eta(\nabla_{\overline{Y}^{V}}\phi\overline{X})
 -2 (A((\phi\overline{X})^{V},\xi, \overline{Y})+A(\theta(\overline{X}^{V}), \xi, \phi\overline{Y})). \nonumber
\end{align}
Then, $\mathbf{v}$ is a symmetric operator if
\begin{eqnarray}
    A(\theta(\overline{X}^{V}), \xi, \phi\overline{Y}) + A(\theta(\phi\overline{X}^{V}), \xi,\overline{Y})=
    A(\theta(\overline{Y}^{V}), \xi, \phi\overline{X}) + A(\theta(\phi\overline{Y}^{V}), \xi, \overline{X}).\nonumber
   \end{eqnarray}
This completes the proof.
\end{proof}
\begin{lemma}\label{dch}
 For a contact metric Finsler structure $(\phi,\eta,\xi,g)$, the covariant derivative of $\xi$ in the direction
 of $X \in \Gamma(TTM_{0})$ is given by:
 \begin{equation}
  \nabla_{X}\xi= -\phi\mathbf{h}(\pi_{\ast}X)-\phi\pi_{\ast}X+ A^{\sharp}(\theta(X),\xi,\bullet)-2A(\theta(X),\xi,\xi)\xi
 \end{equation}
where, 
\begin{equation}
 g(A^{\sharp}(\theta(X),\xi,\bullet),\overline{Z})= A(\theta(X),\xi, \overline{Z}),
\end{equation}
for all $\overline{Z} \in \Gamma(\pi^{\ast}TM)$.
\end{lemma}
\begin{proof}
 By Lemma \ref{cordcx}, we have
 \begin{align}
  &2g((\nabla_{X}\phi)\xi,\overline{Z}) =  g(\mathcal{N}_{H}^{(1)}(\xi,\overline{Z}), \phi\pi_{\ast}X)-2 d^{H}\eta(\phi\overline{Z}, \pi_{\ast}X)
  -2 A(\theta(X),\xi,\phi\overline{Z})\cr
  &= -g((\mathcal{L}_{\xi^{H}}\phi)\overline{Z}, \pi_{\ast}X)-2g(\overline{Z},\pi_{\ast}X)+ 2g(\overline{Z},\eta(\pi_{\ast}X)\xi)-2
  A(\theta(X),\xi,\phi\overline{Z})\cr
  &= -2g(\mathbf{h}\overline{Z},\pi_{\ast}X)-2g(\overline{Z},\pi_{\ast}X)+ 2g(\overline{Z},\eta(\pi_{\ast}X)\xi)-2
  A(\theta(X),\xi,\phi\overline{Z}).\nonumber
  \end{align}
Thus, by symmetry of $\mathbf{h}$ we have
\begin{eqnarray}
 \phi\nabla_{X}\xi= \mathbf{h}\pi_{\ast}X+ \pi_{\ast}X- \eta(\pi_{\ast}X)\xi+ \phi A^{\sharp}(\theta(X), \xi,\bullet).\nonumber
\end{eqnarray}
Applying $\phi$ to this equation, one has
\begin{equation}
 \nabla_{X}\xi= -\phi\mathbf{h}(\pi_{\ast}X)-\phi\pi_{\ast}X+ A^{\sharp}(\theta(X),\xi,\bullet) -2A(\theta(X),\xi,\xi)\xi\nonumber
\end{equation}
which completes the proof.
\end{proof}
\begin{corollary}\label{ach}
Let $(\phi,\eta,\xi,g)$ be a contact metric structure. Considering the identification $\pi^{\ast}TM\cong \mathcal{H}TM_{0}$,  the operator $\mathbf{h}$  anti-commutes with $\phi$ and
 \begin{equation}
  trace_{g}\mathbf{h}= 0.
 \end{equation}
\end{corollary}
\begin{proof}
 For all $\overline{X}, \overline{Y} \in \Gamma(\pi^{\ast}TM)$, we have
 \begin{align}
  2g(\overline{X},\phi\overline{Y})&=  g(\nabla_{\overline{X}^{H}}\xi, \overline{Y})-g(\nabla_{\overline{Y}^{H}}\xi, \overline{X})\cr
  &= g(-\phi\mathbf{h}(\overline{X})-\phi\overline{X}, \overline{Y})
  -g(-\phi\mathbf{h}(\overline{Y})-\phi\overline{Y}, \overline{X})\cr
  &=  2g(\overline{X},\phi\overline{Y}) - g(\phi\mathbf{h}(\overline{X}), \overline{Y}) +
  g(\phi\mathbf{h}(\overline{Y}), \overline{X}).
 \end{align}
Consequently,
\begin{eqnarray}
 - g(\phi\mathbf{h}(\overline{X}), \overline{Y}) +
  g(\phi\mathbf{h}(\overline{Y}), \overline{X})= 0,
\end{eqnarray}
and we obtain
\begin{equation}
 \mathbf{h}\phi + \phi\mathbf{h}=0
\end{equation}
So $\mathbf{h}$ anti-commutes with $\phi$. For the last assertion, if $\mathbf{h}\overline{X}= \lambda \overline{X}$, then
\begin{equation}
 \mathbf{h}\phi\overline{X} =   -\phi\mathbf{h}\overline{X} = -\lambda \phi\overline{X}.
\end{equation}
Thus if $\lambda$ is an eigenvalue of $\mathbf{h}$, so is $- \lambda$ and hence $ trace_{g}\mathbf{h}=0$.
\end{proof}

Now, we discuss some aspects of the hh-curvature $R$, with respect to the Chern connection of contact metric Finsler
pulled-back bundle $(\pi^{\ast}TM, \phi, \eta, \xi,g)$.

\begin{proposition}\label{proric}
 On contact metric Finsler pulled-back bundle  $(\pi^{\ast}TM, \phi, \eta, \xi,g)$, we have:
 \begin{eqnarray}\label{cxh}
  \left(\nabla_{\xi^{H}}\mathbf{h}\right)\pi_{\ast}X =  \phi\pi_{\ast}X -\mathbf{h}^{2}\phi \pi_{\ast}X+ \phi R(\xi^{H},X)\xi +
\phi A^{\sharp}(\theta [\xi^{H}, X], \xi, \bullet),
 \end{eqnarray}
and
\begin{align}\label{cxh1}
 & \frac{1}{2}\left(R(\xi^{H},X)\xi- \phi R(\xi^{H}, (\phi \pi_{\ast}X)^{H})\xi \right) = \mathbf{h}^{2} \pi_{\ast}X
  + \frac{1}{2}\phi^{2}A^{\sharp}(\theta([\xi^{H}, X]), \xi, \bullet)\cr
 &\quad + \phi^{2}\pi_{\ast}X  + \frac{1}{2}\phi A^{\sharp}(\theta([\xi^{H}, (\phi\pi_{\ast}X)^{H}]),\xi,\bullet).
\end{align}
\end{proposition}
\begin{proof}
The fact that $\nabla_{\xi^{H}}\xi= 0$, we have
 \begin{eqnarray}
  R(\xi^{H}, X)\xi &=& \nabla_{\xi^{H}}\nabla_{X}\xi- \nabla_{[\xi^{H}, X]}\xi.\nonumber
 \end{eqnarray}
Then, by Lemma \ref{dch} we obtain
\begin{eqnarray}\label{hhc}
  R(\xi^{H}, X)\xi &=& -\phi\nabla_{\xi^{H}}\mathbf{h}(\pi_{\ast}X)-\phi \nabla_{\xi^{H}}\pi_{\ast}X+ \phi \mathbf{h}(\pi_{\ast}[\xi^{H}, X])
 + \phi\pi_{\ast}[\xi^{H}, X]\cr && +  A^{\sharp}(\theta([\xi,^{H}X]),\xi,\bullet)-2A(\theta([\xi,^{H}X]),\xi,\xi)\xi.
 \end{eqnarray}
 Applying $\phi$ to the equation (\ref{hhc}), and using the fact that $\nabla_{\xi^{H}}\phi=0$ 
 we have,
 \begin{align}
  \phi R(\xi^{H}, X)\xi &=  \nabla_{\xi^{H}}(\pi_{\ast}X + \mathbf{h}\pi_{\ast}X)- \eta(\nabla_{\xi^{H}}(\pi_{\ast}X+
  \mathbf{h}\pi_{\ast}X))\xi -\mathbf{h} \pi_{\ast}[\xi^{H},X] \cr
  &  + \eta(\pi_{\ast}[\xi^{H},X])\xi -\pi_{\ast}[\xi^{H},X]+\phi A^{\sharp}(\theta ([\xi^{H}, X], \xi, \bullet)\cr
  &=  \nabla_{\xi^{H}}(\pi_{\ast}X + \mathbf{h}\pi_{\ast}X)-\mathbf{h} \pi_{\ast}[\xi^{H},X] -\pi_{\ast}[\xi^{H},X]\cr
  &+\phi A^{\sharp}(\theta ([\xi^{H}, X], \xi, \bullet)\nonumber
 \end{align}
  By Corollary \ref{ach} and the fact that 
 $\eta(\nabla_{\xi^{H}}\mathbf{h}\pi_{\ast}X= 0$, we obtain
  \begin{align}\label{hccf}
  \phi R(\xi^{H}, X)\xi &=  (\nabla_{\xi^{H}}\mathbf{h})\pi_{\ast}X - \phi \pi_{\ast}X + \mathbf{h}^{2}\phi \pi_{\ast}X \cr &  -
  \phi A^{\sharp}(\theta ([\xi^{H}, X], \xi, \bullet),
 \end{align}
which leads to the relation (\ref{cxh}).  Now, from (\ref{hccf}) we have
 \begin{align}\label{for1}
 \phi R(\xi^{H}, ( \phi(\pi_{\ast}X))^{H})\xi
 &= (\nabla_{\xi^{H}}\mathbf{h})\phi(\pi_{\ast}X) - \phi^{2}(\pi_{\ast}X) + \mathbf{h}^{2}\phi^{2}\pi_{\ast}X \cr &  -
  \phi A^{\sharp}(\theta ([\xi^{H},(\phi(\pi_{\ast}X))^{H}], \xi, \bullet)\cr
  &=  -\phi(\nabla_{\xi^{H}}\mathbf{h})(\pi_{\ast}X ) - \mathbf{h}^{2}\pi_{\ast}X-\phi^{2}\pi_{\ast}X \cr & -
  \phi A^{\sharp}(\theta ([\xi^{H}, (\phi(\pi_{\ast}X))^{H}], \xi, \bullet),
 \end{align}
and,
\begin{align}\label{for2}
 R(\xi^{H}, X)\xi &=  -\phi((\nabla_{\xi^{H}}\mathbf{h})\pi_{\ast}X)+\phi^{2}(\pi_{\ast}X)+\mathbf{h}^{2}\pi_{\ast}X  \cr
  & - \phi^{2}(A^{\sharp}(\theta([\xi^{H}, X], \xi, \bullet)).
\end{align}
Subtracting (\ref{for2}) by (\ref{for1}), we get (\ref{cxh1}).
\end{proof}
Recall that the pulled-back bundle $\pi^{\ast}TM$ is locally of dimension $2n+1$. Then, as in Riemannian case, we have a $\phi$-basis
of $\pi^{\ast}TM$ given by
\begin{equation}
          \left\{\partial_{i},\;\; \partial_{i^{\ast}}:= \phi\partial_{i},\;\; \xi\right\}_{i=1\cdots n}.
  \end{equation}
It is well know  that (see \cite{mb2} for more details), by the trace of hh-curvature of Chern connection, we have a Finslerian analogous of Ricci tensor, called
horizontal Ricci (1,1;0)-tensor, denoted $Ric^{H}$ and given by 
\begin{equation}\label{tra}
 Ric^{H}(\overline{Z},X)= trace_{g}( \overline{Y}\mapsto R(X,\overline{Y}^{H})\overline{Z}),
\end{equation}
where $X\in \Gamma(TTM_{0})$ and $\overline{Y}, \overline{Z} \in \Gamma(\pi^{\ast}TM)$.
 \begin{proposition}\label{hkcs}
   On contact metric Finsler pulled back bundle  $(\pi^{\ast}TM, \phi, \eta, \xi,g)$, the horizontal Ricci $(1,1;0)$-tensor curvature
   with respect to     $\xi$ and $\xi^{H}$, is given by:
   \begin{align}
    Ric^{H}(\xi,\xi^{H}) & = - 2n + 2 trace_{g}\mathbf{h}^{2}  + \sum_{i=1}^{n}\frac{1}{2}\left\{A(\theta([\xi^{H}, \delta_{i}]), \xi,\partial_{i})\right.\nonumber\\
    &\left.-A(\theta([\xi^{H}, \delta_{i^{\ast}}]), \xi,\phi\partial_{i})\right\},
   \end{align}
where $\delta_{i}= \partial_{i}^{H}$ and $\delta_{i^{\ast}}= \partial_{i^{\ast}}^{H}$, for all $i=1,\cdots,n$.
 \end{proposition}
 \begin{proof}
  By (\ref{tra}), and the fact that $R(\xi^{H}, \xi^{H})\xi=0$ we have
  \begin{eqnarray}
   Ric^{H}(\xi, \xi^{H})&=&\sum_{i=1}^{n}\left[g(R(\xi^{H},\delta_{i})\xi, \partial_{i})
   +g(R(\xi^{H},\delta_{i^{\ast}})\xi, \partial_{i^{\ast}})\right]\cr
   &=& \sum_{i=1}^{n}\left[g(R(\xi^{H},\delta_{i})\xi-\phi R(\xi^{H},\phi\delta_{i})\xi, \partial_{i})\right].\nonumber
  \end{eqnarray}
Then, by Proposition \ref{proric}, we obtain
\begin{align}
   Ric^{H}(\xi, \xi^{H})&= \sum_{i=1}^{n}2g\left( \mathbf{h}^{2}\partial_{i}
 + \phi^{2}\partial_{i} + \frac{1}{2} \phi^{2}A^{\sharp}(\theta([\xi^{H}, \delta_{i}]), \xi, \bullet), \partial_{i}\right)\cr
 &  +\sum_{i=1}^{n}g\left(\phi A^{\sharp}(\theta([\xi^{H}, \phi \delta_{i}],\xi,\bullet)), \partial_{i}\right),\nonumber
   \end{align}
 which proves the assertion.
 \end{proof}
 \begin{theorem}
 A contact metric Finsler pulled back bundle  $(\pi^{\ast}TM, \phi, \eta, \xi,g)$ is horizontally k-contact if and only if
  \begin{equation}
    Ric^{H}(\xi,\xi^{H}) = - 2n + \sum_{i=1}^{n}\left(A(\theta([\xi^{H}, \delta_{i}]), \xi,\partial_{i})
   -A(\theta([\xi^{H}, \delta_{i^{\ast}}]), \xi,\phi\partial_{i})\right).\nonumber
   \end{equation}
 \end{theorem}
\begin{proof}
The proof is derived from Proposition \ref{hkcs} and the fact that, $\mathbf{h}=0$ for the horizontal $K$-contact metric Finsler structures.
\end{proof}
In case of the flag curvature \cite{AA}, we have a Finslerian
analogous part of Hatakeyama, Ogawa and Tanno's result \cite{HOT}.
\begin{theorem}
Let $(\pi^{\ast}TM, \phi, \eta, \xi,g)$ be a horizontally K-contact metric Finsler pulled back bundle. Then, the flag curvature with the transverse edge $\xi$ is equal to $1$.
\end{theorem}
\begin{proof}
Assume that $(\pi^{\ast}TM, \phi, \eta, \xi,g)$ is horizontally
K-contact then $\mathbf{h}= 0$. On the other hand, the flag
curvature with the transverse edge $\xi$ is given by
\begin{align}
K(l,\xi)&=  g(R(l^{H}, \xi^{H})\xi, l) =  g (\phi
\nabla_{\xi^{H}}\mathbf{h}l-\phi^{2}l-\mathbf{h}^{2}l-\phi^{2}A^{\sharp}(\theta[\xi^{H},l^{H}],\xi,\bullet),l )\nonumber\\
&= -g(\nabla_{\xi^{H}}\mathbf{h}l,\phi l)+ 1- g(\mathbf{h}^{2}l,l),\nonumber
\nonumber
\end{align}
and the assertion follows by hypothesis.
\end{proof}

\section*{Acknowledgments}

The second author would like to thank the University of KwaZulu-Natal for financial support.

 \bibliographystyle{plain}

\end{document}